\documentclass{article}

\usepackage{latexsym,amsmath,amssymb,epic}
\usepackage{latexsym,amssymb,epic}
\usepackage{amsmath,amsthm}
\usepackage{color}

\usepackage{latexsym,amssymb,epic}
\usepackage{amsmath,amsthm}
\usepackage{color}

\newtheorem{theorem}{Theorem}%[section]
\newtheorem{lemma}{Lemma}%[section]
\newtheorem{corollary}{Corollary}%[section]
\begin{document}

\title{Parity Considerations for Mex-Related Partition Functions of Andrews and Newman}

\author{Robson da Silva and James A. Sellers}

\date{}
\maketitle

\begin{abstract}
In a recent paper, Andrews and Newman extended the mex-function to integer partitions and proved many partition identities connected with these functions. In this paper, we present parity considerations of one of the families of functions they studied, namely $p_{t,t}(n)$. Among our results, we provide complete parity characterizations of $p_{1,1}(n)$ and $p_{3,3}(n)$.
\end{abstract}



\section{Introduction}
\label{SectIntro}

In \cite{A-N}, Andrews and Newman generalized the minimal excludant function (mex-function) to apply to integer partitions. Given a partition $\lambda$ of $n$, they defined the mex-function $\mbox{mex}_{A, a}(\lambda)$ to be the smallest integer congruent to $a$ modulo $A$ that is not part of $\lambda$. Denoting the number of partitions $\lambda$ of $n$ satisfying  
$$\mbox{mex}_{A,a}(\lambda) \equiv a \pmod{2A}$$
by $p_{A, a}(n)$, they proved that the generating function for $p_{t,t}(n)$ (see \cite[Lemma 9]{A-N}) is given by
\begin{equation}
\displaystyle\sum_{n=0}^{\infty} p_{t,t}(n) q^n = \frac{1}{(q;q)_{\infty}} \sum_{n=0}^{\infty} (-1)^n q^{tn(n+1)/2},
\label{GFk} 
\end{equation}
where we use the following standard $q$-series notation:
$$\begin{array}{l} 
(a;q)_0 = 1, \\
(a;q)_n = (1-a)(1-aq) \cdots (1-aq^{n-1}),\  \forall n \geq 1, \\ 
\end{array}$$
and
$$\begin{array}{l} 
(a;q)_{\infty} = \lim_{n \to \infty} (a;q)_n,\  |q|<1.
\end{array}$$

The two functions $p_{1,1}(n)$ and $p_{3,3}(n)$ play an important role in \cite{A-N}. Indeed, Andrews and Newman proved the following connections between these two functions and the enumeration of partitions according to their rank and their crank.

\begin{theorem}[\cite{A-N}, Theorem 2] 
For all $n\geq 1,$	$p_{1,1}(n)$ equals the number of partitions of $n$ with non-negative crank.
\end{theorem} 
\begin{theorem}[\cite{A-N}, Theorem 3] 
For all $n\geq 1,$	$p_{3,3}(n)$ equals the number of partitions of $n$ with rank $\geq -1$.
\end{theorem} 

In this paper, we present parity results for $p_{t,t}(n)$ for several odd values of $t.$ In Section \ref{Sect2}, we present complete parity characterizations for $p_{1,1}(n)$ and $p_{3,3}(n)$ which are available to us thanks to the lacunarity of the corresponding generating functions modulo 2.  Unfortunately, the generating functions for $p_{t,t}(n)$ modulo 2 for $t>3$ fail to be as lacunary.  Even so, a set of congruences modulo 2 for $p_{t,t}(2kn+j)$, $t=5, 7, 11, 13, 17, 19, 23$, is presented in Section \ref{SectT3}.  All of the proof techniques required to prove these parity results are elementary, including classical generating function manipulations and well--known results of Euler and Jacobi.

%%%%%%%%%%%%%%%%%%%%%%%%%%%%%%%%%%%%%%%%%%%%%%%%%%%%%%%%%%%%%%

\section{Parity characterizations of $p_{1,1}(n)$ and $p_{3,3}(n)$}
\label{Sect2} 

In order to prove the parity characterization of $p_{1,1}(n)$, we need the following well-known identities.

\begin{lemma} The following identities hold true:
\begin{align}
(q;q)_{\infty}^{3} &  = \sum_{n=0}^{\infty} (-1)^n (2n+1)q^{n(n+1)/2}, \label{Jacobi} \\
(q;q)_{\infty} & =  \sum_{n=-\infty}^{\infty} (-1)^n q^{n(3n-1)/2} \label{Euler} 
\end{align}
\end{lemma}

\begin{proof}
Equation \eqref{Jacobi} is Jacobi's identity (see \cite[Eq. (1.3.24)]{Berndt}).  Equation \eqref{Euler} is Euler's Pentagonal Number Theorem (see \cite[Eq. (1.3.18)]{Berndt}). 
\end{proof}

%\subsection{Proof of Theorem \ref{T1}}

With the above in hand, we now present the characterization modulo 2 for $p_{1,1}(n).$  

\begin{theorem} 
\label{1characterization}
For all $n \geq 1$,
	$$p_{1,1}(n) \equiv \left\{ 
	\begin{array}{l}
	1 \pmod{2}, \mbox{ if } n = k(3k\pm 1) \mbox{\ for some\ } k, \\
	0 \pmod{2}, \mbox{\ otherwise.}
	\end{array}  \right. $$
	\label{T1} 
\end{theorem}

\begin{proof}
Taking $k=1$ in \eqref{GFk}, we find that
\begin{align}
\displaystyle\sum_{n=0}^{\infty} p_{1,1}(n) q^n & = \frac{1}{(q;q)_{\infty}} \sum_{n=0}^{\infty} (-1)^n q^{n(n+1)/2} \nonumber \\ & \equiv \frac{1}{(q;q)_{\infty}} \sum_{n=0}^{\infty} q^{n(n+1)/2} \pmod{2}.
\label{GF1mod2} 
\end{align}
From Jacobi's identity \eqref{Jacobi}, we have
\begin{equation*}
(q;q)_{\infty}^{3} \equiv \sum_{n=0}^{\infty} q^{n(n+1)/2} \pmod{2}.
\end{equation*}
Using this fact and \eqref{GF1mod2} we obtain an expression for the generating function for $p_{1,1}(n)$ modulo $2$:
\begin{align}
\displaystyle\sum_{n=0}^{\infty} p_{1,1}(n) q^n \equiv (q;q)_{\infty}^2 \equiv (q^2;q^2)_{\infty}  \pmod{2}.
\end{align}

In order to complete the proof, we make use of Euler's Pentagonal Number Theorem \eqref{Euler} to obtain
\begin{align*}
\displaystyle\sum_{n=0}^{\infty} p_{1,1}(n) q^n & \equiv (q^2;q^2)_{\infty}  \equiv \sum_{k=-\infty}^{\infty} q^{k(3k-1)} \pmod{2} \\
& \equiv \sum_{k=1}^{\infty} q^{k(3k+1)} + \sum_{k=0}^{\infty} q^{k(3k-1)} \pmod{2},
\end{align*}
which completes the proof.
\end{proof}

With Theorem \ref{1characterization} in hand, we can now prove a number of corollaries for specific arithmetic progressions.  

\begin{corollary}
For all $n\geq 0,$ $p_{1,1}(2n+1) \equiv 0 \pmod{2}.$  
\end{corollary}

\begin{proof}
Note that $2n+1$ can never be represented as $k(3k \pm 1)$ because $k(3k \pm 1)$ is always even (it is twice a pentagonal number for any value $k$).  Thus, by Theorem \ref{1characterization}, the result follows.  
\end{proof}

\begin{corollary}
Let $p\geq 5$ be prime, and let $r$, $1\leq r\leq p-1$ be such that $12r+1$ is a quadratic non--residue modulo $p.$  Then, for all $n\geq 0,$ $p_{1,1}(pn+r) \equiv 0 \pmod{2}.$  
\end{corollary}

\begin{proof}
Note that, if $pn+r = k(3k \pm 1)$ for some $k,$ then $12(pn+r)+1 =  (6k\pm 1)^2.$  Notice also that
$$
12(pn+r)+1 = 12pn +12r+1 \equiv 12r+1 \pmod{p}
$$
but $12r+1$ is assumed to be a quadratic non--residue modulo $p.$  Therefore, such a representation cannot exist.  So by Theorem \ref{1characterization}, $p_{1,1}(pn+r) \equiv 0 \pmod{2}$ for all $n\geq 0.$  
\end{proof}

%\subsection{Proof of Theorem \ref{T2}}

We discuss now the parity characterization of $p_{3,3}(n)$. By \eqref{GFk} and Ramanujan's theta function
\begin{equation}
\psi(q) = \sum_{n=0}^{\infty} q^{n(n+1)/2} = \frac{(q^2;q^2)_{\infty}^{2}}{(q;q)_{\infty}},
\label{psi}
\end{equation}
we have
\begin{align*}
\displaystyle\sum_{n=0}^{\infty} p_{3,3}(n) q^n & \equiv \frac{1}{(q;q)_{\infty}} \sum_{n=0}^{\infty} q^{3n(n+1)/2} \pmod{2} \\
&= \frac{1}{(q;q)_\infty}\frac{(q^6;q^6)_\infty^2}{(q^3;q^3)_\infty}  \ \ \ \ \ \ \ \mbox{ (by \eqref{psi})} \\
&\equiv \frac{1}{(q;q)_\infty}\frac{(q^3;q^3)_\infty^4}{(q^3;q^3)_\infty} \pmod{2} \\
& = \frac{(q^3;q^3)_{\infty}^3}{(q;q)_{\infty}}.
\end{align*}
The generating function for the number of 3-core partitions (see \cite[Theorem 1]{H-S}), denoted by $a_3(n)$, is given by
\begin{align*}
\sum_{n=0}^{\infty} a_3(n)q^n = \frac{(q^3;q^3)_{\infty}^3}{(q;q)_{\infty}}.
\end{align*}
So, $p_{3,3}(n) \equiv a_3(n) \pmod{2}$ for all $n \geq 0$. Thanks to the work of Hirschhorn and Sellers on a closed formula for $a_3(n),$ (see \cite[Theorem 6]{H-S}), we have the following parity characterization for $p_{3,3}(n):$ 

\begin{theorem}
\label{3characterization}
 For all $n \geq 1$,
	$$p_{3,3}(n) \equiv \left\{ 
	\begin{array}{l}
	1 \pmod{2}, \mbox{ if } 3n+1 \mbox{\ is a square}, \\
	0 \pmod{2}, \mbox{\ otherwise.}
	\end{array} \right. $$
	\label{T2} 
\end{theorem}

This characterization can be applied rather easily to prove several infinite families of parity results for $p_{3,3}(n).$  We share a number of these corollaries here.  

\begin{corollary}
\label{cor33-1}
For all $m\geq 0$ and all $n\geq 0,$ 
$$p_{3,3}\left(4^{m+1}n+\frac{7\cdot 4^m-1}{3}\right) \equiv 0 \pmod{2}$$ 
and 
$$p_{3,3}\left(4^{m+1}n+\frac{10\cdot 4^m-1}{3}\right) \equiv 0 \pmod{2}.$$ 
\end{corollary}

\begin{proof}
Note first that 
$$
3\left(4^{m+1}n+\frac{7\cdot 4^m-1}{3}\right) + 1 = 4^m(12n+7) 
$$ 
after straightforward simplification.  Note that $4^m$ is a square while $12n+7$ cannot be.  (This is clear since $12n+7 \equiv 3 \pmod{4}$ and all squares are either $0$ or $1$ modulo 4.)  Thus, 
$$
3\left(4^{m+1}n+\frac{7\cdot 4^m-1}{3}\right) + 1 
$$ 
is not a square, and the result is proved thanks to Theorem \ref{3characterization}.

Next, note that 
$$
3\left(4^{m+1}n+\frac{10\cdot 4^m-1}{3}\right) + 1 = 4^m(12n+10) 
$$ 
after straightforward simplification.  Note that $4^m$ is a square while $12n+10$ cannot be.  (This is clear since $12n+10 \equiv 2 \pmod{4}$ and all squares are either $0$ or $1$ modulo 4.)  Thus, 
$$
3\left(4^{m+1}n+\frac{10\cdot 4^m-1}{3}\right) + 1 
$$ 
is not a square, and the result is proved thanks to Theorem \ref{3characterization}. 
\end{proof}

\begin{corollary}
\label{cor33-2}
For all $m\geq 0$ and all $n\geq 0,$ 
$$p_{3,3}\left(2\cdot 4^{m+1}n+\frac{ 13\cdot 4^m-1}{3}\right) \equiv 0 \pmod{2}.$$ 
\end{corollary}

\begin{proof}
Note that 
$$
3\left(2\cdot 4^{m+1}n+\frac{ 13\cdot 4^m-1}{3}\right)  + 1 = 4^m(24n+13) 
$$ 
after straightforward simplification.  Note  that $4^m$ is a square while $24n+13$ cannot be.  (This is clear since $24n+13 \equiv 5 \pmod{8}$ and all squares are either $0, 1$ or $4$ modulo 8.)  Thus, 
$$
3\left(2\cdot 4^{m+1}n+\frac{13\cdot 4^m-1}{3}\right)  + 1
$$ 
is not a square, and the result is proved thanks to Theorem \ref{3characterization}.  
\end{proof}

\begin{corollary}
Let $p\geq 5$ be prime, and let $r$, $1\leq r\leq p-1$ be such that $3r+1$ is a quadratic non--residue modulo $p.$  Then, for all $n\geq 0,$ $p_{3,3}(pn+r) \equiv 0 \pmod{2}.$  
\end{corollary}

\begin{proof}
Note that $3(pn+r)+1$ can never be square.  This is because 
$$
3(pn+r)+1 = 3pn +3r+1 \equiv 3r+1 \pmod{p}
$$
and $3r+1$ is assumed to be a quadratic non--residue modulo $p.$   So by Theorem \ref{3characterization}, $p_{3,3}(pn+r) \equiv 0 \pmod{2}$ for all $n\geq 0.$  
\end{proof}

%%%%%%%%%%%%%%%%%%%%%%%%%%%%%%%%%%%%%%%%%%%%%%%%%%%%%%%%%%%%%%%%%%%

\section{Additional congruences}
\label{SectT3}
We now consider parity results for $p_{t,t}(n)$ for $t\geq 5.$  While characterizations modulo 2 for these functions do not appear to be readily available, we can still prove a significant set of Ramanujan--like congruences for several of these functions.  

\begin{theorem} 
\label{RSTheorem}
For all $n \geq 0$,
	\begin{align*}
	p_{5,5}(10n+j) & \equiv 0 \pmod{2},  j \in \{2, 6\},  \nonumber \\ %\label{Cong1} \\
	p_{7,7}(14n+j) & \equiv 0 \pmod{2},  j \in \{7, 9, 13 \},  \nonumber \\ %\label{Cong2} \\
	p_{11,11}(22n+j) & \equiv 0 \pmod{2},  j \in \{2, 8, 12, 14, 16\},  \nonumber \\ %\label{Cong3}  \\  
	p_{13,13}(26n+j) & \equiv 0 \pmod{2},  j \in \{2, 10, 16, 18, 20, 22 \},  \nonumber \\ % \label{Cong4}  \\
	p_{17,17}(34n+j) & \equiv 0 \pmod{2},  j \in \{11, 15, 17, 19, 25, 27, 29, 33 \},  \nonumber \\ %\label{Cong5}  \\
	p_{19,19}(38n+j) & \equiv 0 \pmod{2},  j \in \{2, 8, 10, 20, 24, 28, 30, 32, 34 \},  \nonumber \\ %\label{Cong6}  \\
	p_{23,23}(46n+j) & \equiv 0 \pmod{2},  j \in \{11, 15, 21, 23, 29, 31, 35, 39, 41, 43, 45 \}.  \nonumber %\label{Cong7} 
	\end{align*}
	% \label{T3}
\end{theorem}

{
\begin{proof}
Let $t \geq 5$ be an odd number.  Taking \eqref{GFk} modulo 2, we have
\begin{align}
\displaystyle\sum_{n=0}^{\infty} p_{t,t}(n) q^n & \equiv \frac{1}{(q;q)_{\infty}} \sum_{n=0}^{\infty} q^{tn(n+1)/2} \pmod{2} \nonumber \\
& \equiv \frac{(q^t;q^t)_{\infty}^3}{(q;q)_{\infty}} \pmod{2}. \ \ \ \ \ \ \ \mbox{ (by \eqref{psi})} \label{eq4}
\end{align}
On the other hand, the generating function for $t$-core partitions is given by
\begin{align*}
\sum_{n=0}^{\infty} a_t(n)q^n = \frac{(q^t;q^t)_{\infty}^t}{(q;q)_{\infty}}.
\end{align*}
Thus, from (\ref{eq4}) we know 
\begin{align*}
\displaystyle\sum_{n=0}^{\infty} p_{t,t}(n) q^n 
& \equiv \frac{(q^t;q^t)_{\infty}^3}{(q;q)_{\infty}} \pmod{2} \nonumber \\
& = \frac{1}{(q^t;q^t)^{t-3}_{\infty}} \displaystyle\sum_{n=0}^{\infty} a_t(n) q^n . 
	\end{align*}
Since $t$ is odd, we then know 
\begin{align}
\sum_{n=0}^{\infty} p_{t,t}(n)q^n \equiv \frac{1}{(q^{2t};q^{2t})_{\infty}^{(t-3)/2}} \sum_{n=0}^{\infty} a_t(n)q^n \pmod{2}.
\label{eq6} 
\end{align}
The $t$-dissection of \eqref{eq6} yields, for each $r \in \{0,1,\ldots, t-1\}$,
\begin{align*}
\sum_{n=0}^{\infty} p_{t,t}(tn + r)q^{n} \equiv \frac{1}{(q^{2};q^{2})_{\infty}^{(t-3)/2}} \sum_{n=0}^{\infty} a_t(tn+r)q^n \pmod{2}.
\end{align*}
Now, after 2-dissecting both sides of the last expression we are left with
\begin{align}
\sum_{n=0}^{\infty} p_{t,t}(2tn + r)q^{n} & \equiv \frac{1}{(q;q)_{\infty}^{(t-3)/2}} \sum_{n=0}^{\infty} a_t(2tn+r)q^n \pmod{2}, \label{eq7} \\
\sum_{n=0}^{\infty} p_{t,t}(2tn + t+r)q^{n} & \equiv \frac{1}{(q;q)_{\infty}^{(t-3)/2}} \sum_{n=0}^{\infty} a_t(2tn+t+r)q^n \pmod{2}. \label{eq8}
\end{align}
Lastly, thanks to Radu and Sellers (\cite[Theorem 1.4]{R-S}) we know that, for all $n\geq 0,$ 
\begin{align*}
a_5(10n+j) & \equiv 0 \pmod{2}, j \in \{2, 6\}, \\
a_7(14n+j) & \equiv 0 \pmod{2}, j \in \{7, 9, 13\}, \\
a_{11}(22n+j) & \equiv 0 \pmod{2}, j \in \{2, 8, 12, 14, 16\}, \\
a_{13}(26n+j) & \equiv 0 \pmod{2},  j \in \{2, 10, 16, 18, 20, 22 \},  \\
a_{17}(34n+j) & \equiv 0 \pmod{2},  j \in \{11, 15, 17, 19, 25, 27, 29, 33 \},    \\
a_{19}(38n+j) & \equiv 0 \pmod{2},  j \in \{2, 8, 10, 20, 24, 28, 30, 32, 34 \}, \\
a_{23}(46n+j) & \equiv 0 \pmod{2},  j \in \{11, 15, 21, 23, 29, 31, 35, 39, 41, 43, 45 \}.  
\end{align*}
This, combined with \eqref{eq7} and \eqref{eq8}, implies the results of this theorem.  
\end{proof}
}

%%%%%%%%%%%%%%%%%%%%%%%%%%%%%%%%%%%%%%%%%%%%%%%%%%%%%%%%%%%%%%%%%%%

\section{Closing thoughts}
%\label{}

Several remarks are in order as we close.  First, in \cite{A-N}, Andrews and Newman introduce a second family of functions which they denote by $p_{2t,t}(n).$  (The reader is referred to \cite{A-N} for the details of this family of functions.) Our hope was that this family would also satisfy various Ramanujan--like congruence properties.  Unfortunately, based on extensive computations, this does not appear to be the case.  

Secondly, we close with two potential paths for future work.  Namely, it would be nice to have combinatorial proofs of these parity results.  It would also be gratifying to have a fully elementary proof of Theorem \ref{RSTheorem} (since Radu and Sellers relied on modular forms to carry out their proof of their parity results for $t$--core partition functions which were mentioned in the proof of Theorem \ref{RSTheorem}).  
%%%%%%%%%%%%%%%%%%%%%%%%%%%%%%%%%%%%%%%%%%%%%%%%%%%%%%%%%%%%%%%%%%%%%%%%%%%%%%%%

\section*{Acknowledgment}

The first author was supported by FAPESP (grant no. 2019/14796-8).

\

\noindent Universidade Federal de S\~ao Paulo, Av. Cesare M. G. Lattes, 1201, S\~ao Jos\'e dos Campos, SP, 12247--014, Brazil. \\
E-mail address: silva.robson@unifesp.br

\

\noindent Department of Mathematics and Statistics, University of Minnesota Duluth, Duluth, MN  55812, USA. \\
E-mail address: jsellers@d.umn.edu


\begin{thebibliography}{99}
\bibitem{A-N} Andrews, G. E. and Newman, D.: {\it The Minimal Excludant in Integer Partitions}. J. Integer Seq. {\bf 23} (2020), Article 21.2.3. 
\bibitem{Berndt} Berndt, B. C.: Number Theory in the Spirit of Ramanujan, AMS, Providence, 2006. 
%\bibitem{Fine} Fine, N. J.: {\it Basic Hypergeometric Series and Applications}. Mathematical Surveys and Monographs, vol. 27, AMS, 1988.
%\bibitem{H1} Hirschhorn, M. D.: {On the parity of $p(n)$}. J. Combin. Theory, Ser. A, {\bf62} (1993), 128--138.
\bibitem{H-S} Hirschhorn, M. D. and Sellers, J. A.: {\it Elementary proofs of various facts about 3-cores}. Bull. Aust. Math. Soc. {\bf 79} (2009), 507--512.
\bibitem{R-S} Radu, S. and Sellers, J. A.: {\it Parity results for broken $k$-diamond partitions and ($2k+1$)-cores}. Acta Arith. {\bf 146} (2011), 43--52.
%\bibitem{Hirschhorn} Hirschhorn, M. D.: {The power of $q$}. Developments in Mathematics, Springer, 2017.
\end{thebibliography}
\end{document}